\documentclass[11pt]{article}
\usepackage{amsmath}
\usepackage{amssymb}
\usepackage{amsthm}

\newtheorem{theorem}{Theorem}[section]
\newtheorem{lemma}[theorem]{Lemma}
\newtheorem{corollary}[theorem]{Corollary}

\theoremstyle{definition}
\newtheorem{definition}[theorem]{Definition}

\theoremstyle{remark}
\newtheorem{remark}[theorem]{Remark}

\numberwithin{equation}{section}

\newcommand{\comp}{B^{\complement}}

\newcommand{\dist}{\mathrm{dist}}

\newcommand{\spam}{\mathop{\mathrm{span}}}

\newcommand{\supp}[1]{{\mathrm{supp}}({#1})}
\newcommand{\ints}{\mathbb{Z}}
\newcommand{\reals}{\mathbb{R}}

\newcommand{\dif}{\mathrm{d}}

\renewcommand{\r}{\rho_0}
\newcommand{\Ain}{\mathcal{A}_{in}}
\newcommand{\Aout}{\mathcal{A}_{out}}
\newcommand{\B}{\mathcal{B}}
\newcommand{\J}{\mathcal{J}}
\renewcommand{\L}{\mathcal{L}}

\begin{document}

\title{The Penalized Lebesgue Constant for Surface Spline Interpolation
}
\author{Thomas Hangelbroek
\thanks{Department of Mathematics, Texas A\&M
    University College Station, TX 77843, USA. Research supported by an NSF Postdoctoral Research Fellowship.}}





\maketitle
\begin{abstract}
Problems involving approximation from scattered data where data is arranged quasi-uniformly have been treated by RBF methods for decades. Treating data with spatially varying density has not been investigated with the same intensity, and is far less well understood. In this article we consider the stability of surface spline interpolation (a popular type of RBF interpolation) for data with nonuniform arrangements. 
Using techniques similar to \cite{HNW}, which discusses interpolation from quasi-uniform data on manifolds, we show that surface spline interpolation on $\reals^d$ is stable, but in a stronger, local sense. We also obtain pointwise estimates showing that the Lagrange function decays very rapidly, and at a rate determined by the local spacing of datasites. These results in conjunction with a Lebesgue lemma, show that surface spline interpolation enjoys the same rates of convergence as those of the local approximation schemes recently developed by DeVore and Ron in \cite{DeRo}.
\end{abstract}



\section{Introduction}
Since their introduction as tools for approximation and interpolation in the 1970s in works of Duchon \cite{D2,D1} and Meinguet \cite{Mein}, {\em surface splines} (known also as polyharmonic splines) have become one of the most prominent examples of {\em radial basis functions} (RBFs). The hallmark of the RBF methodology is to compute with
linear combinations of scattered translates of a basic function (the RBF): 
$$s = \sum_{\xi\in \Xi} A_{\xi} \phi(x-\xi),
$$
where the linear combination is taken over a finite set of {\em centers} $\Xi$
(although often one supplements a polynomial of low degree to such a linear combination).
By choosing the RBF correctly, one can effectively treat high dimensional data by, e.g., interpolation, least squares minimization, or Tikhonov regularization to name a few popular 
{\em schemes}. The  surface splines  are defined by
$$\phi(x):=\phi_m(x) := \begin{cases} |x|^{2m-d},&\quad \text{when $d$ is odd};\\
             |x|^{2m-d} \log|x|,&\quad \text{when $d$ is even}.
            \end{cases}$$
In this article, we investigate error and stability of surface spline interpolation, 
two subjects which still harbor 
 many open questions. 

Of principal interest are questions concerning norms of the operators underlying the approximation scheme, especially when those operators are projectors
(as is the case for interpolation and $L_2$ minimization), because such norms measure the stability of the approximation process. When the approximation scheme is interpolation -- the scheme investigated in this article --  the $L_{\infty}$ operator norm is called the {\em Lebesgue constant}.
Finding bounds for such norms based on the computational cost 
(measured as the dimension of the space of approximants) is a pervasive problem in approximation theory, despite the fact that there are few examples where such constants can be satisfactorily bounded. 
Univariate spline theory provides an important exception, with  operator norms 
being uniformly bounded; results along these lines hold for interpolation, $L_2$ minimization and many other schemes -- with a bound independent of the complexity of the problem (i.e., regardless of the number of knots), cf. \cite{deBoor, Demko, Shad}.
For polynomial interpolation, interpolation by trigonometric polynomials on the circle, or spherical harmonics on $d$-dimensional spheres, Lebesgue constants grow with the dimension of the space of interpolants, indicating that interpolation in these settings is unstable. 
We develop and investigate a ``penalized'' Lebesgue constant, a quantity that measures stability of the interpolation process under highly nonuniform distributions of centers.

The results we present in this article show that under some simple assumptions, surface spline interpolation of compactly supported functions is stable, but, more than this, it is 
stable in a very strong sense, one that takes into account the local distribution of centers. 

In the course of our analysis, we also show that the Lagrange (or fundamental) functions associated with 
this kind of interpolation are very well localized:  
they decay nearly exponentially in terms of the spatial variable, and this decay 
becomes dramatically more rapid in regions of high density -- the Lagrange function also scales linearly with the local spacing of the centers. 
This is a significant accomplishment in its own right, as it demonstrates that the space of RBFs has a basis that is fairly local. The importance of this stems from the fact that RBFs are globally supported  and have influence far from their center. Even considering RBFs with compact support does little to diminish this fact; because one does not dilate the RBF as the centers become dense, the RBF captures more and more centers in its support as the density increases. The Lagrange functions do, however, adapt to changing density of the centers.

Another result we present concerns the error incurred by interpolation. Recently, DeVore and Ron \cite{DeRo} have developed error estimates for surface spline approximation that are both {\em precise} (with convergence in $L_p$ correctly dictated by the $L_p$ smoothness of the function) and {\em local} (with pointwise convergence controlled by the local density).  In the larger context of approximation theory, precision is not so remarkable, but for kernel based approximation, precise results have been notoriously difficult to prove. With few exceptions, for RBF approximation such results are known only when error is measured in $L_2$, and only then for target functions residing in certain reproducing kernel Hilbert spaces. The estimates in \cite{DeRo} are local in the sense that the error is also controlled by the local density of the centers. Such results, which have long been known for univariate spline approximation, allow faster convergence in regions of high density.

The key arguments used in our results have been developed gradually, 
over a period of more than 30 years. 
At the heart of our approach is a discrete Gr{\" o}nwall's inequality employed by Descloux \cite{Des} and Douglas, Dupont and Wahlbin \cite{DDW}.
A similar inequality was used by Matveev in \cite{Mat1},
 to obtain results about interpolation by $D^m$splines, which are interpolating 
functions satisfying certain variational problems.
As in the RBF case, such interpolants can be expressed as a linear combination of kernels,
but the kernels depend strongly on the geometry of the underlying domain $\Omega$. 
Because of this, they are not generally well suited for direct computation. 
But when $\Omega = \reals^d$, this corresponds to the problem of surface spline interpolation on $\reals^d$, and provides an ideal starting point for our setting. 
Indeed, Matveev's work plays an important role in Johnson's  \cite{Johnson} resolution of
the $L_p$ saturation order for surface spline 
interpolation on bounded domains when $1\le p\le2$. 

Very recently, this approach has been used to bound Lebesgue constants for interpolation on Riemannian manifolds by Narcowich, Ward, and the author \cite{HNW}. 
In fact, it is shown in \cite{HNW} that when the manifold is $\reals^d$
and centers are assumed to be quasi-uniform,
the boundedness of the Lebesgue constant 
for surface spline and `Sobolev spline' interpolation follows rather 
directly from the results of Matveev \cite{Mat1}. 
Whereas the challenge in the manifold setting is to construct kernels suitable for the type of argument developed for $D^m$ splines, and then to formulate such an argument to work in different geometries,
the crucial difference in this article is to find such an argument that  treats centers with spatially varying density (as opposed to quasi-uniform centers whose density can be measured by a  single, spatially fixed parameter). 
\subsection*{Organization}
In the following section, we introduce the penalized Lebesgue constants and describe their role in the stability of the interpolation problem we consider. We also discuss the Lebesgue lemma,
one of the main motivations for this investigation. In Section 3, we discuss the main technical lemmas needed to provide sharp bounds for the Lagrange functions. Section 4 provides Sobolev and pointwise decay estimates of the Lagrange functions. The main results -- the bound on the Lebesgue constant and the use of the  Lebesgue lemma to obtain optimal local approximation rates for surface spline interpolation are given in Section 5.
%
\subsection*{Background}
In this paper, we treat the interpolation of functions on an open, bounded domain $\Omega\subset \reals^d$. We denote the diameter of $\Omega$ by $r_1:=\sup_{x,y\in \Omega} \dist(x,y).$ The functions we interpolate are compactly supported in $\Omega$, with  
$$r_0:= \dist(\supp f, \partial \Omega)>0.$$ 
The set of nodes for interpolation coincide with the set of centers $\Xi \subset \Omega$ (introduced previously), and this is assumed 
to be finite. Moreover,
the subset $\Xi_f = \Xi \cap \supp f$ is assumed to be nonempty. A consequence of this assumption is that the notorious `boundary effects' are avoided\footnote{The presence of a boundary poses an obstacle in nearly all computational problems; approximation with RBFs is no exception to this, although the negative effects of a boundary and how to overcome these is gradually being understood, cf. \cite{H1}. 
}.

The space of polynomials of degree less than or equal to $n$ is denoted by $\Pi_{n}$. The set of centers $\Xi$
is said to be {\em unisolvent} with respect to $\Pi_n$ if the only $p\in\Pi_n$ that vanishes on $\Xi$ is the zero
polynomial.

We are now ready to define the surface spline interpolant.
Given a continuous function $f\in C(\reals^d)$, a set of centers $\Xi$ unisolvent with respect to $\Pi_{m-1}$, and a surface spline $\phi=\phi_m$
with $m>d/2$, the surface spline interpolant to $f$ is the function $I_{\Xi}f$ that satisfies $I_{\Xi}f(\xi) = f(\xi)$ for all $\xi\in\Xi$,
and that resides in the set
\begin{equation}\label{space}
S_{\Xi} := 
\left\{\sum_{\xi \in \Xi} A_{\xi} \phi(\cdot -\xi)+ p(\cdot)
\mid 
p\in \Pi_{m-1},\  
\sum_{\xi\in\Xi} A_{\xi}q(\xi) = 0, \ 
\forall q\in \Pi_{m-1}
\right\}.
\end{equation}
Existence and uniqueness of the surface spline interpolant is a consequence of the conditional positive
definiteness of the function $\phi_m$, which, in turn, is a consequence of the positivity of its generalized Fourier transform. 
See \cite[Chapter 8]{Wend} for the relevant background.

There is a dual, variational characterization of surface spline interpolation which states that the interpolant $I_{\Xi}f$ is the minimizer 
 of the $L_2$ Sobolev seminorm of order $m$ over all interpolants to $f$ at the set $\Xi$ (cf. \cite[Chapter 10]{Wend}). 
 In other words, $|I_{\Xi}f|_{W_2^m(\reals)} \le |u|_{W_2^m(\reals^d)}$ for all $u$ for which $u|_{\Xi}= f|_{\Xi}$.
 The Sobolev
 seminorm is defined as follows.  
\begin{definition}\label{Def:Sobolev}
For an open set $\Upsilon\subset \reals^d$, we define the $m^{\mathrm{th}}$ Sobolev seminorm:
\begin{equation}\label{def_ssn}
| f|_{W_2^m(\Upsilon)}:= \left( \sum_{|\beta| = m}  \frac{m!}{\beta!} \int_{\Upsilon}  |D^{\beta} f(x)|^2\, \dif x\right)^{1/2}.
\end{equation}
\end{definition}

The space of locally integrable functions having finite $m^{\mathrm{th}}$ seminorm is the homogeneous Sobolev space, denoted $\mathring{W}_2^m(\Upsilon)$. 
When $\Upsilon = \reals^d$, we suppress the dependence on $\Upsilon$, writing simply $\mathring{W}_2^m$, and express the seminorm as $|f|_m$. 
In the literature this is also frequently referred to as the Beppo-Levi space $BL^m$. 
When $m>d/2$, the homogeneous Sobolev space is embedded in the space of continuous functions, and is, therefore, a reproducing kernel semi-Hilbert space. On occasion we also use the (inhomogeneous) Sobolev norm 
$\| f\|_{W_2^m(\Upsilon)}:= 
(\sum_{k=0}^m | f|_{W_2^k(\Upsilon)}^2)^{1/2}$.
 
Throughout this article we also consider Sobolev seminorms when $\Upsilon$ is a  ball, the complement of a ball, or an  annulus. 
For $R>0$ we denote the open ball centered at $x$ having radius $R$ by $B(x,R)$. 
We denote its complement by $\comp(x,R) = \reals^d \backslash B(x,R)$.
For $R,w>0$, we define the annulus centered at $x$ having outer radius $R$ and width $w$ by
\begin{equation}\label{annulus}
A(x,R,w):= B(x,R) \backslash   B(x,R-w) = \comp(x,R-w) \backslash \comp(x,R)
\end{equation}
where equality is modulo sets of measure zero. 
When $x=0$, we will suppress the first argument in the above notations. 
\section{Penalized Lebesgue Constants and Surface Spline Interpolation}
A motivating goal is
to provide estimates of the form
\begin{equation}\label{norm}
\left\|\frac{I_{\Xi}f}{\rho^{\sigma}}\right\|_{\infty}
\le 
C_{\sigma}\left\|\frac{f}{\rho^{\sigma}}\right\|_{\infty},\qquad\text{for}\,f\in C(\Omega)
\end{equation}
where $\rho:\Omega\to (0,\infty)$ is a positive function measuring, roughly, the density
of the set of centers $\Xi$: $\rho(x)$  indicates the spacing of centers near to $x$. We are mostly interested in situations where the density of $\Xi$ is permitted to vary considerably.  This is done without the usual assumption of {\em quasi-uniformity},
$$h := \max_{x\in\Omega}\dist(x,\Xi) \le c_0q, \qquad \text{where}\; q:= \min_{\xi\in\Xi}\dist(\xi,\Xi\setminus\{\xi\})$$ an assumption in effect for most interpolation results; in fact, this is most interesting when the centers are deliberately in violation of the quasi-uniformity assumption. 
By allowing points to coalesce in such a way that the density of points varies spatially, we find that the notion of the global fill distance $h$ as an approximation parameter is simply too crude. At some locations we will have spacing of points much smaller than $h$.

On one hand, an estimate of the form (\ref{norm}) shows the stability of the interpolation process: 
setting $\sigma = 0$ shows that the interpolant is never much greater than the original function, while choosing $\sigma>0$ gives a refined notion of stability, putting greater
emphasis on regions where $\rho$ is small. 

On the other hand, this leads to a Lebesgue lemma that respects the local distribution of centers, one where we can interpolate the error from any other surface spline approximant:
$$
\left\|\frac{f-I_{\Xi}f}{\rho^{\sigma}}\right\|_{\infty}
= 
\left\|\frac{f-s-I_{\Xi}(f-s)}{\rho^{\sigma}}\right\|_{\infty}
\le 
(1+C_{\sigma})\inf_{s\in S_{\Xi}}\left\|\frac{f -s}{\rho^{\sigma}}\right\|_{\infty}.
$$
Here $S_{\Xi}$ is the finite dimensional space of surface spline approximants, defined in (\ref{space}).
This shows interpolation is ``near best'' in the sense of local approximation. 
I.e., that the best local approximation rates are attainable for interpolation as well. DeVore and Ron  \cite{DeRo} have recently investigated local surface spline approximation, providing a scheme $f\mapsto s_{f,\Xi}$ that delivers precise local approximation results of the form $|f(x) -s_{f,\Xi}(x)|\le C{\rho(x)^{\sigma}}\|f\|_{C^{\sigma}(\Omega)}$, for $0<\sigma \le 2m$.

We denote by $\chi_{\xi}$ the unique function in the homogeneous Sobolev space 
$\mathring{W}_2^m$
equaling $1$ at $\xi$, vanishing on $\Xi \backslash \{\xi\}$ and having minimal 
semi-norm $|\cdot |_{W_2^m(\reals^d)} $ over all functions sharing these properties. 
That is, $\chi_{\xi}$ is the Lagrange function for surface spline interpolation
centered at $\xi\in \Xi$. 
The Lagrange functions permit us to rewrite the  interpolant as
$I_{\Xi} f = \sum_{\xi \in \Xi} f(\xi) \chi_{\xi}$.

Our goal is to bound the operator norm of the interpolation operator on weighted 
$L_{\infty}$ spaces. 
By a a simple application of H{\"o}lder's inequality we observe that
$$
\left| \frac{I_{\Xi}f(x)}{\rho(x)^{s}} \right| 
=
\left| \frac{\sum_{\xi\in\Xi}f(\xi)\chi_{\xi}(x)}{\rho(x)^{s}} \right|
\le
\left(
   \sup_{\xi\in \Xi}
  \left|\frac{f(\xi)}{\rho(\xi)^{s}}\right| 
\right)
\left(
  \sum_{\xi\in\Xi} 
   \left|\chi_{\xi}(x)
  \left(\frac{\rho(\xi)}{\rho(x)}\right)^{s}\right|
\right).
$$
Assuming a `slow growth assumption' introduced in \cite{H} and defined in the following section, it suffices to bound
$
\L_{\sigma,\rho} 
:= 
\sup_{x\in \Omega}\sum_{\xi\in\Xi}  
|\chi_{\xi}(x)| 
\left(1+\frac{|x-\xi|}{\rho(x)}\right)^{\sigma}
$
for any nonnegative $\sigma = s (1-\epsilon)$.
For a target function $f$ having support compactly contained in $\Omega$ 
(as we assume throughout this article) 
the summation is over the subset $\Xi_f$, so we need only bound
\begin{equation}\label{pen_leb_const_def}
\L_{\sigma,\rho} 
= 
\sup_{x\in \Omega}\sum_{\xi\in\Xi_f}  |\chi_{\xi}(x)| \left(1+\frac{|x-\xi|}{\rho(x)}\right)^{\sigma}.
\end{equation}
We note that, when $\sigma = 0$ (or when $\rho$ is constant), this is precisely the usual Lebesgue constant, and for this reason, we call $\L_{\sigma,\rho}$ the {\em penalized Lebesgue constant}.
We approach this by estimating the Lagrange functions associated with $\xi$ slightly removed from the boundary of $\Omega$. The result we obtain, under assumptions placed on the centers, detailed in the subsequent sections, (Theorem \ref{Theorem:Lagrange}) is
$$|\chi_{\xi} (x)| \le C  \left(1+\frac{|x-\xi|}{\rho(\xi)}\right)^\sigma\exp \left[-\lambda\left(\frac{|x-\xi|}{\rho(\xi)}\right)^\epsilon\right],$$ 
for all $x$ in the largest ball centered at $\xi$ that is contained in $\Omega$.

\section{Local Density Functions}
We begin by developing the notion of a local density (LD) parameter. 
Such parameters are key to the results of \cite{DeRo}, although
we follow the approach taken in \cite{H}.
Initially, we can think of the LD as a function on $\Omega$ which, at $\alpha\in\Omega$, is the minimal radius needed to capture a $K$-stable local polynomial reproduction of order $\ell$. 
\begin{definition}[Local Density Parameter] \label{Def:LocalDensity}
Given a set of centers $\Xi \subset \reals^d$, a \emph{local density parameter (LD)}  $\rho:\reals^d \to \reals_+$ is a function with an associated {\it local polynomial reproduction}
$a:(\xi,\alpha)\mapsto a(\xi,\alpha):\Xi\times \reals^d \to \reals$ of precision $\ell$:
\vspace{-.3ex}
\begin{description}
\item[\bf  (Support)]  For $|\xi - \alpha|>\rho(\alpha)$, $a(\xi,\alpha) = 0$.
\item[\bf (Precision)] For all $p\in \Pi_{\ell}$ we have $\sum_{\xi\in\Xi} a(\xi,\alpha)p(\xi) = p(\alpha)$.
\item[\bf (Stability)] There is $K>0$ such that $\sum_{\xi\in\Xi} |a(\xi,\alpha)|<K$ for all $\alpha$.
\end{description}
\end{definition}
This definition is not quite sufficient for our purposes, and we impose an extra condition of global compatibility
to ensure that the LD does not grow too rapidly.
This is the approach taken in \cite{H} and which is at the heart of DeVore and Ron's error estimates. 
\begin{definition}[Slow Growth]
We say $\rho$ exhibits $1-\epsilon$ {\em slow growth} if
\begin{equation}\label{sg}
\rho(\alpha) \le \rho(x)\left(1+\frac{|x-\alpha|}{\rho(x)}\right)^{1-\epsilon},
\end{equation}
for every $x$ and $\alpha$ in $\Omega$.
\end{definition}
The function $\rho:\Omega\to \reals_{+}$ exhibits {\em self-majorization} of order $\tau$ if there is a constant $C_{sm}>0$ so that for every $x$ and $y$ in $\Omega$, we have
\begin{eqnarray}
\label{sm}
\rho(y) \ge C_{sm} \rho(x)\left(1+\frac{|x-y|}{\rho(x)}\right)^{-\tau}.
\end{eqnarray}
In \cite{H} it has been shown that these conditions, (\ref{sg}) and (\ref{sm}), are actually equivalent assumptions, provided $\tau+1=1/\epsilon$. If
$\rho$ satisfies $1-\epsilon$ slow growth, then it satisfies self-majorization of order $\tau$, with constant $C_{sm}$ depending on $\epsilon$ only.
\subsection*{Norming Sets and the Zeros Lemma}
The norming set concept has become an essential tool 
in scattered data approximation since being introduced 
in this context in \cite{Jet}. 
In our setting, it allows us to transfer the $L_{\infty}$ norm of a polynomial
 on a domain (such as a ball) to the norm of its restriction to a  subset.
A set $\Upsilon$ is a norming set for a ball $B$ with {\em norming constant} $\kappa$
if for polynomials $p\in \Pi_{\ell}$ we have
\begin{equation}\label{norming}
\|p\|_{L_{\infty}(B)}\le \kappa \|p_{|_{\Upsilon}}\|_{\ell_{\infty}(\Upsilon)}.
\end{equation}
 A set of centers associated with a local density function provides precisely this:
%
%
\begin{lemma}\label{Lemma:Norming_Set}
Suppose $\Xi$ has a LD $\rho$ with precision $\ell$, stability $K$ that satisfies slow growth. Then $\Xi_x:=\Xi\cap B(x,3\rho(x))$ is a norming set for polynomials $p\in \Pi_{\ell}$ in the ball $B(x, 3\rho(x))$. 
\end{lemma}
\begin{proof}
The maximum value of $\rho(y)$ for $y\in B(x,\rho(x))$ is bounded by $2\rho(x)$, so for every $y\in B(x,\rho(x))$ there
exists a local polynomial reproduction $a(y,\xi)$ of order $\ell$ having support inside $B(y,2\rho(x))\subset B(x,3\rho(x))$. Thus
$$|p(y)|\le \left|\sum_{\xi \in \Xi} a(\xi,y) p(\xi)\right| \le  \left(\sum_{\xi\in \Xi} \left|a(\xi,y)\right|\right)
\, \max_{\xi \in B(y,\rho(y))} |p(\xi)| \le K \|p_{|_{\Xi_x}}\|_{\ell(\Xi_x)}.$$
\end{proof}
%
%
This brings us to the main result of this section: the zeros lemma, which states that the norm of a function in a Sobolev space, vanishing on a sufficiently large set, is bounded by a Sobolev seminorm of that function.
%
%
\begin{lemma}[Zeros] \label{Lemma:Zeros}
Suppose  $x \in\reals^d$ , $ m>d/2$,and $r>0$. If $\Upsilon\subset B(x,3r)$ 
is a norming set for $p\in\Pi_{m-1}$ with norming constant $\kappa$,
then there is a constant $C$ depending on $m, \kappa$ and $d$ such that if 
$f\in W_2^{m}(B(x,3r))$ 
vanishes on $\Upsilon$, then
$$
\sum_{k\le m}
r^{2(k-m)} |f|_{W_2^k\bigl(B(x,r)\bigr)}^2 
\le 
C |f|_{W_2^m\bigl(B(x,3r)\bigr)}^2.
$$
\end{lemma}
\begin{proof}
We begin by assuming $r=1$, $x=0$, and writing $B_1 = B(0,1)$ and $B_3=B(0,3)$. 
For any $f\in W_2^m(B_3)$, 
there is a polynomial $p$ of degree at most $m-1$ so that 
$\|f-p\|_{W_2^m(B_3)}\le C |f|_{W_2^m(B_3)}$ 
(this is sometimes called the Bramble-Hilbert Lemma). 
We can estimate $\|p\|_{W_2^m(B_1)}$ by
its $L_{\infty}(B_1)$ norm, with a constant depending only on $d$ and $m$, and, consequently by the $\ell_{\infty}$ norm incurring a constant $\kappa$,
$$
\|p\|_{W_2^m(B_1)}
\le 
C \|p\|_{L_{\infty}(B_1)} 
\le 
C \kappa \|p_{|_{\Upsilon}}\|_{\ell_{\infty}(\Upsilon)}.
$$
By our assumption, $f$ vanishes on $\Upsilon$, so there is no cost in subtracting its restriction from $p_{|_{\Upsilon}}$. 
Thus, we get
\begin{eqnarray*}
\|p\|_{W_2^m(B_1)} 
\le 
C \|(p-f)_{|_{\Upsilon}}\|_{\ell_{\infty}(\Upsilon)}
\le 
C \|p-f\|_{L_{\infty}(B_3)} 
\le 
C \|p-f\|_{W_2^m(B_3)} .
\end{eqnarray*}
The last inequality is an application of the Sobolev embedding theorem. It follows that
$$\|f\|_{W_2^m(B_1)} \le \|f-p\|_{W_2^m(B_1)} +\|p\|_{W_2^m(B_1)} \le (1+C)|f|_{W_2^m(B_3)}.$$
The result for general $r$ and $x$ holds by applying an affine change of variable, $y \mapsto x+ry$.
\end{proof}

\begin{remark}\label{rmrk}
Property (\ref{norming}) can be modified to treat slightly smaller subsets -- a fact we use later. 
In particular, if (\ref{norming}) holds for $\Pi_{\ell}$, 
then a similar result holds for polynomials of lower degree
by removing one point from $\Upsilon$. Let $\Upsilon' =  \Upsilon \setminus \{\xi\}$, and consider
$p\in\Pi_{\ell -1}$. We wish to show that 
\begin{equation*}
\|p\|_{L_{\infty}(B)}\le 
\kappa C_{\ell} 
\left\|
  p_{|_{(\Upsilon' )}}
\right\|_{\ell_{\infty}(\Upsilon')}.
\end{equation*}
Without loss of generality, we may consider $B = B(0,1)$ and $\Upsilon \subset B(0,3)$, 
since the result for general radii follows by rescaling.

There is a constant $c_{\ell}$ depending only on $\ell$ (but independent of $\xi$) so that 
$\|p\|_{L_{\infty}(B)}\le c_{\ell} \|p\|_{L_{\infty}(\Omega)}$ where $\Omega = B\setminus B(\xi,1/2). $
By compactness, we can choose $y\in \mathrm{cl}(B)$ so that 
$|y-\xi|>1/2$ and $c_{\ell}|p(y)|>\|p\|_{B}$. 
Consider $P \in\Pi_{\ell}$, where 
$P(z) := \langle(z-\xi),\frac{y-\xi}{|y-\xi|}\rangle p(z)$.
Then 
$$
\|p\|_{L_{\infty}(B)} 
\le 
2 c_{\ell}  |y-\xi| \, \bigl|p(y)\bigr|
= 2c_{\ell}  |P(y)|  
\le
2  c_{\ell} \kappa \left\|P|_{\Upsilon}\right\|_{\ell_{\infty}(\Upsilon)} .$$
The last inequality follows by applying (\ref{norming}) to $P$.

Observe that $P(\xi) =0$, so 
$ \left\|P|_{\Upsilon}\right\|_{\ell_{\infty}(\Upsilon)} =  \left\|P|_{\Upsilon'}\right\|_{\ell_{\infty}(\Upsilon')} $. 
Since $|\xi - \zeta|\le 6$ for every $\zeta\in \Upsilon'$, it follows that
$\|p\|_{L_{\infty}(B)}\le 2c_{\ell} \kappa \|P\|_{\ell_{\infty}(\Upsilon')} \le 
12  c_{\ell} \kappa \|p\|_{\ell_{\infty}(\Upsilon)}$.
\end{remark}
%
%
%
%
\section{Lagrange function estimates}
%
%
To show that the Lagrange function is rapidly decaying, 
we make use of a type of discrete Gr{\" o}nwall lemma. 
The key estimate is to show that the bulk of the tail of an exponentially decaying function can be captured over a  sufficiently wide annulus. For an increasing function $R:\reals_+\to\reals_+$ we express this as:
$$|\chi_{\xi}|_{W_2^m\bigl(  B(\xi,R(t)) \backslash B(\xi,R(t-1))\bigr)} \ge 
(1-\mu) |\chi_{\xi}|_{W_2^m\bigl(\reals^d \backslash B(\xi,R(t-1))\bigr)}.$$
 for $0<\mu<1$ independent of $t$. This implies the complementary inequality,
$|\chi_{\xi}|_{W_2^m\bigl( \reals^d\backslash B(\xi, R(t))\bigr)} \le 
\mu |\chi_{\xi}|_{W_2^m\bigl(\reals^d \backslash B(\xi,R(t-1))\bigr)},$
which can be iterated $\lceil t-1\rceil$ times to obtain
$$|\chi_{\xi}|_{W_2^m\bigl( \reals^d\backslash B(\xi, R(t))\bigr)} \le 
C\mu^t |\chi_{\xi}|_{W_2^m(\reals^d)},$$
This is the type of argument made in Lemma \ref{Lemma:Bulk} and Lemma \ref{Lemma:Tail}. 
%
\begin{lemma}\label{Lemma:Bulk}
Assume that $\Xi\subset \Omega$ has an LD $\rho$, with precision $m-1$ satisfying $1-\epsilon$ slow growth. 
Suppose, furthermore, that $\dist(\xi,\partial \Omega)<r_0$.  
Then there exists a constant $\mu\in (0,1)$ 
(depending on $\epsilon, m, d$ and the stability constant $K$ of $\rho$) 
so that for sufficiently small $\rho(\xi)$, 
i.e., for $\rho(\xi)\le \gamma:= (\epsilon r_0^{\mbox{}-\epsilon})^{\frac{1}{1-\epsilon}}$,
and
$3<t<\left(r_0/\rho(\xi)\right)^\epsilon$,
$$|\chi_{\xi}|_{W_2^m\bigl(\reals^d \backslash B\bigl(\xi,\rho(\xi)t^{1/\epsilon}\bigr)\bigr)} \le \mu |\chi_{\xi}|_{W_2^m\bigl(\reals^d \backslash B\bigl(\xi,\rho(\xi)[t-3]^{1/\epsilon}\bigr)\bigr)}.$$
\end{lemma}
\begin{proof} 
Without loss of generality, we assume $\xi = 0$ and write $\r:=\rho(\xi)$ and $\chi:=\chi_{\xi}$.
With $1+\tau =1/\epsilon$, construct a $C^{\infty}$ cutoff $\phi$, vanishing outside 
$B((t-1)^{1+\tau}\r)$ 
and equaling $1$ on 
$B((t-2)^{1+\tau}\r)$. 
That is, $\phi$ is piecewise constant outside of the annulus
$$
\Ain 
:= 
A\left((t-1)^{1+\tau}\r, \tilde{w}(t)\right) 
$$ 
which has outer radius $(t-1)^{1+\tau}\r$ and
width $\tilde{w} := \bigl[(t-1)^{1+\tau}-(t-2)^{1+\tau}\bigr] \r $. 
We note that $(\tau+1)(t-2)^{\tau}\r\le \tilde{w}(t)\le(\tau+1)(t-1)^{\tau}\r\le1$, for 
$\r$  less than  $\gamma$.
It follows from the chain rule that 
\begin{equation}\label{chainrule}
\|D^{\alpha} \phi\|_{\infty} \le C(\alpha) (\tilde{w})^{-|\alpha|}.  
\end{equation}
By the variational property of $\chi$, we have that
$$
|\chi|_{W_2^m(\reals^d)}^2
\le 
  | \phi \chi|_{W_2^m(\reals^d)}^2
\le  
  |\chi|_{W_2^m\left(\B \right)}^2 
  + 
  |\phi \chi|_{W_2^m(\Ain)}^2,
$$
where we express the ball interior to $\Ain$ by the compact notation:
$$\B: =  B\left(\r(t-2)^{1+\tau}\right).$$
Our goal for the remainder of the proof is to estimate the quantity $|\phi \chi|_{W_2^m(\Ain)}^2.$
This takes the form of two estimates. 
The first estimate is concerned with the effect of multiplying
by the cutoff, which can be easily measured using Leibniz's rule in conjunction with (\ref{chainrule}). 
This is independent of the radius $t$, and allows
us to control the seminorm in terms of a biased Sobolev norm of $\chi$. The second estimate will control
this biased Sobolev norm via partitioning the annulus into small balls where Lemma \ref{Lemma:Zeros} can
be applied.

{\bf Estimate 1} The desired result is to control 
$|\phi \chi|_{W_2^m\left(\Ain \right)}^2$
by a weighted combination of Sobolev seminorms of $\chi$ on $\Ain.$ 
We observe first of all that we can simplify matters
by applying the product rule:
\begin{eqnarray*}
 |\phi \chi|_{W_2^m(\Ain)}^2 
&=& 
\sum_{|\alpha|= m} 
\int_{\Ain} 
  \frac{m!}{\alpha!} 
  \left| 
    \sum_{\beta\le \alpha} 
      {\alpha \choose \beta} 
      D^{\beta} \chi(x) D^{\alpha-\beta}\phi(x)
  \right|^2 
\dif x\nonumber\\
&\le& 
C_1 
\sum_{|\alpha|\le m} \sum_{\beta\le \alpha}
  \int_{\Ain} 
    \left| 
      D^{\beta}\chi(x) D^{\alpha-\beta}\phi(x) 
    \right|^2 
  \dif x. \nonumber
\end{eqnarray*}
The constant $C_1 = C_1(m,d)$ is seen to depend only on $m$ and $d$.
We continue with the estimate, now observing that we can remove the factors 
$D^{\alpha - \beta}\phi(x)$ by (\ref{chainrule}).
Thus, 
\begin{eqnarray}\label{weightednorm}
|\phi \chi|_{W_2^m(\Ain)}^2
&\le& 
C_2  \sum_{|\alpha|\le m} 
     \sum_{\beta\le \alpha}
     \int_{\Ain}
        \bigl(  \tilde{w} \bigr)^{2(|\beta|-|\alpha|)}
        \left|  D^{\beta} \chi (x) \right|^2
     \dif x
\nonumber\\
&\le& 
C_3  \sum_{k\le m}   
       \bigl(\tilde{w}\bigr)^{2(k-m)}
       | \chi |_{W_2^k(\Ain)}^2. 
\end{eqnarray}
 The first inequality follows by applying (\ref{chainrule}) and the observation
 that any two multi-indices $\beta<\alpha$ satisfy $|\alpha-\beta| = |\alpha|-|\beta|$. 
The second inequality is obtained by rearranging terms.

{\bf Estimate 2} We are nearly in a position to apply Lemma \ref{Lemma:Zeros}. Cover $\Ain$ with a sequence of balls $(B_j)_{j\in \J}$ such that
\begin{itemize}
\item each ball is of radius $3\r t^{\tau}$,
 \item each ball has its center in $\Ain$ and, thus, $B_j\subset \Aout$
 where
 $$ 
\Aout:=
B\bigl(\r\bigl[(t-1)^{1+\tau} + 3t^{\tau}\bigr]\bigr)
\setminus 
B\bigl(\r\bigl[(t-2)^{1+\tau} - 3t^{\tau}\bigr]\bigr),
$$
\item the balls have a finite intersection property: every $x\in \Aout$ is in at most $N_d$ balls $B_j$, with $N_d$ depending
only on the spatial dimension and not on $t,\r, \tau$, etc. 
\end{itemize}
Writing the center of each ball $B_j$ as $c_j$, 
we observe that the slow growth assumption implies that 
the density $\rho(c_j)$ is not greater than 
$\r \bigl(1+(t-1)^{1+\tau}\bigr)^{1-\epsilon} \le \r t^{\tau}$, 
and Lemma \ref{Lemma:Zeros} may be applied on each ball.
The expression from (\ref{weightednorm}) can be bounded by corresponding norms
carried by the balls $B_j$:
\begin{eqnarray*}
\sum_{k\le m}
  \bigl(\tilde{w}\bigr)^{2(k-m)}
  | \chi  |_{W_2^k(\Ain)}^2 
&\le& 
\sum_{k\le m} \sum_{j\in \J} 
  \bigl(\tilde{w}\bigr)^{2(k-m)}
  |\chi|_{W_2^k(B_j)}^2\\
&\le&   \sum_{j\in \J} \sum_{k\le m}
\left(\r(\tau+1) \left(\frac{1}{3} t\right)^{\tau} \right)^{2(k-m)}
|\chi|_{W_2^k(B_j)}^2\\
&\le&  \kappa^2 c^{-2m}\sum_{j\in \J}  
| \chi|_{W_2^m(B_j)}^2
\le  C | \chi|_{W_2^m(\Aout)}^2.
\end{eqnarray*}
The second inequality follows from the observation that, because $t>3$, 
the width of the annulus, which
satisfies 
$\tilde{w} \ge \r(\tau+1)(t-2)^{\tau}$, is greater than
$\r(\tau+1) \left(\frac{1}{3} t\right)^{\tau}$. 
The third inequality follows from Lemma \ref{Lemma:Zeros}, with $c = (\tau+1) \left(\frac{1}{3}\right)^{\tau+1}$  and $\kappa$ the
constant from Lemma \ref{Lemma:Zeros}. The final inequality follows from the finite intersection property, and $C = N_d \kappa^2 c^{-2m}$.
Combining this with (\ref{weightednorm}), we see that
\begin{equation}
|\chi|_{W_2^m(\reals^d \backslash \mathcal{B})}^2 = |\chi|_{W_2^m(\reals^d)}^2-|\chi|_{W_2^m(\mathcal{B})}^2
\le |\phi \chi |_{W_2^m(\Ain)}^2
\le C_4 |\chi|_{W_2^m(\Aout)}^2   \label{sqrnorms}
\end{equation} 
with $C_4 = C_4(m,d)  =C_3 C $. 

It follows from their definitions that 
$\mathcal{B} = B\left(\r (t-2)^{1+\tau} \right)\subset B\left(\r t^{1+\tau} \right)$
which implies 
$\comp\left(\r t^{1+\tau}\right)\subset \mathcal{B}^{\complement}$. In turn, we
can rewrite the annulus as the difference of two punctured planes 
$\Aout \subset \comp\left(\r (t-3)^{1+\tau}\right)\backslash \comp\left(\r t^{1+\tau}\right).$ 
Applying this to (\ref{sqrnorms}) we get
$$
|\chi|_{W_2^m\left(\comp\left(\r t^{1+\tau}\right)\right)}^2\le 
C_4 \left(|\chi|_{W_2^m\left(\comp\left(\r(t-3)^{1+\tau}\right)\right)}^2 - 
|\chi|_{W_2^m\left(\comp\left(\r t^{1+\tau}\right)\right)}^2\right)
$$
and, consequently:
$$
|\chi|_{W_2^m\left(\comp(\r t^{1+\tau} )\right)}
\le 
\dfrac{\sqrt{C_4}}{\sqrt{1+C_4}}
|\chi|_{W_2^m\left(\comp\left( \r(t-3)^{1+\tau}\right)\right)}= 
\mu 
|\chi|_{W_2^m\left(\comp\left( \r(t-3)^{1+\tau} \right)\right)}
$$
with $\mu := \frac{\sqrt{C_4}}{\sqrt{1+C_4}}<1$.
\end{proof}
Iterating this, we show that the seminorm of the characteristic function on a punctured space decays exponentially.
%
%
\begin{lemma}\label{Lemma:Tail}
Under the assumptions of Lemma \ref{Lemma:Bulk}, 
there are constants $C,\nu>0$ depending only on $\epsilon$, $d$, $m$ and $K$
 so that
$$|\chi_{\xi}|_{W_2^m(\comp(\xi,T))} 
\le 
C
|\chi_{\xi}|_{W_2^m(\reals^d)}
\exp\left[
  -\nu \left(\frac{T}{\rho(\xi)}\right)^{\epsilon}
\right] 
$$
holds for $\rho(\xi)\le \gamma$ (the constant from Lemma \ref{Lemma:Bulk}) and $ 3^{\tau+1}\rho(\xi) <T< r_0$.
\end{lemma}
%
\begin{proof}
As before we take $\xi = 0$ and write $\rho(\xi) = \rho(0)= \r$.
Let $T = \r t^{\tau+1}$, and let $N$ be the greatest integer less than 
$\frac{1}{3}\left[\left(\frac{T}{\r}\right)^{\epsilon}-3\right] .$ 
This represents the number of times Lemma \ref{Lemma:Bulk} can be iterated.
Thus, we see that
\begin{eqnarray*}
|\chi|_{W_2^m(\comp(T))} 
&=& 
|\chi|_{W_2^m(\comp(\r t^{\tau+1}))}\\
&\le& \mu |\chi|_{W_2^m( \comp(\r (t-3)^{\tau+1})  )}\\
&\le& \mu^N |\chi|_{W_2^m(\reals^d)}
 \le  
\mu^{\frac{1}{3}\left[\left(\frac{T}{\r}\right)^{\epsilon}-3\right]} |\chi|_{W_2^m(\reals^d)}.
\end{eqnarray*}
The result follows with $e^{-\nu} = \mu^{\frac{1}{3}}$ and $C = \mu^{-1}$.
\end{proof}
Let $q(\xi) := \dist(\xi, \Xi\setminus \{\xi\})$ be the `separation distance' at $\xi$.
By comparing the seminorm of the Lagrange function to that of a suitably scaled bump, it is possible to estimate $|\chi_{\xi}|_m$ in terms of $q(\xi)$, the distance from $\xi$ to its nearest neighbor. On the other hand, by applying the Sobolev embedding lemma for balls, in conjunction with the zeros lemma, Lemma \ref{Lemma:Zeros}, it is possible
to obtain pointwise estimates for $\chi_{\xi}$. This is done in the following theorem.
%
%
\begin{theorem} \label{Theorem:Lagrange}
Let $\Xi\subset \Omega$ denote a finite set of centers with associated LD 
$\rho$ of 
precision $\ell\ge m$ 
and stability $K$, 
satisfying $1-\epsilon$ slow growth. 
Then there is $\lambda>0$ depending on $\epsilon, m, d$ and $K$
(specifically, $\lambda:= \nu/2^{\epsilon}$, with $\nu$ from Lemma \ref{Lemma:Tail})
so that for every $\xi \in \Xi$, with $\dist(\xi,\partial\Omega) \le r_0$
$$
|\chi_{\xi} (x)| 
\le 
C \left(\frac{\rho(\xi)}{q(\xi)}\right)^{m-d/2}
\left(1+\frac{|x-\xi|}{\rho(\xi)}\right)^s\ 
 \exp\left[- \lambda
      \left(
         \frac{\min(|x-\xi|,r_0)}{\rho(\xi)}
      \right)^\epsilon
 \right]
$$
with $s=(1-\epsilon)(m-d/2)$.
\end{theorem}
%
\begin{proof}
We may compare the Lagrange function $\chi_{\xi}$ to another (worse) interpolant $\sigma(\frac{\cdot-\xi}{q(\xi)})$, 
to make the initial observation that  
$$|\chi_{\xi}|_m \le \bigl(q(\xi)\bigr)^{d/2 -m} |\sigma|_m.$$
%
The Sobolev embedding theorem implies that $|u(0)|\le C\|u\|_{W_2^m(B(0,1))}$. Setting
$u(y)  = \chi_{\xi}\bigl(x+\rho(x) y\bigr)$, we have
$$
|\chi_{\xi} (x)|
\le 
C \left(\sum_{k\le m}\rho(x)^{2k-d} |\chi_{\xi}|_{W_2^k\bigl(B(x,\rho(x))\bigr)}^2\right)^{1/2}.$$
By Remark \ref{rmrk},
$\bigl(\Xi\cap B(x,3\rho(x)\bigr)\setminus\{\xi\}$ 
forms a norming set for $\Pi_{m}$ in $B(x,\rho(x))$. 
Thus, Lemma \ref{Lemma:Zeros} implies 
\begin{eqnarray*}
|\chi_{\xi} (x)|
&\le& 
C \rho(x)^{m-d/2}
\left(\sum_{k\le m}
  \rho(x)^{2(k-m)} |\chi_{\xi}|_{W_2^k\bigl(B(x,\rho(x)\bigr)}^2 
\right)^{1/2}\\
&\le& 
C  \rho(x)^{m-d/2}|\chi_{\xi}|_{W_2^m\bigl(\reals^d\bigr)}
\le C \left( \frac{ \rho(x)}{q(\xi)}\right)^{m-d/2}.
\end{eqnarray*}
Applying the slow growth assumption 
$\rho(x)
\le 
\rho(\xi) \left(1+\frac{|x-\xi|}{\rho(\xi)}\right)^{1-\epsilon}$
gives the result for $\rho(\xi)>\gamma$, since in this case
the factor $ \exp\left[- \lambda
      \left(
         \frac{\min(|x-\xi|,r_0)}{\rho(\xi)}
      \right)^\epsilon
 \right]\ge C$, a constant depending only on $r_0,\epsilon$ and $\lambda$.

This also proves the theorem 
for the case $\rho(\xi)\le\gamma$ and 
$|\xi-x|\le k_{\epsilon}\rho(\xi)$, where 
$$k_\epsilon:= 6^{1/\epsilon}2^{(1-\epsilon)/\epsilon}.$$
Suppose $\rho(\xi)\le\gamma$ and $|\xi-x|\ge k_{\epsilon}\rho(\xi)$.
For $T = \dist(\xi,x)$ in the range
$3^{\tau+1}\rho(\xi) \le T \le r_0$, 
we again consider centers in the ball, $B(x,3\rho(x))$, where Lemma \ref{Lemma:Zeros}
can be applied. By a similar argument to the one employed already, we have
\begin{eqnarray*}
|\chi_{\xi}(x)|
&\le& C \rho(x)^{m-d/2} |\chi_{\xi}|_{W_2^m(\comp(\xi,T-3\rho(x)))}\\
&\le& C \bigl(\rho(x)/q(\xi)\bigr)^{m-d/2}
\exp\left[
  -\nu\left(\frac{T-3\rho(x)}{\rho(\xi)}\right)^{\epsilon}
\right]. 
\end{eqnarray*}
Writing $\Gamma := \frac{|x-\xi|}{\rho(\xi)}$, the bracketed statement in the exponent can be written as 
$$\left[-\nu\left(\Gamma -3\frac{\rho(x)}{\rho(\xi)}\right)^{\epsilon}\right]\le 
\left[-\nu\left(\Gamma-3(2\Gamma)^{1-\epsilon}\right)^{\epsilon}\right],$$ 
where we invoke slow growth and the fact that $\Gamma>1$ to estimate $\frac{\rho(x)}{\rho(\xi)}\le (2\Gamma)^{1-\epsilon}.$  
Observe that
$\Gamma> k_{\epsilon} = 6^{1/\epsilon} \cdot 2^{(1-\epsilon)/\epsilon}$,
implies that 
$\Gamma - 3(2\Gamma)^{1-\epsilon}\ge \Gamma/2$, 
so
$[-\nu\left(\Gamma -3\frac{\rho(x)}{\rho(\xi)}\right)]
< 
[-\lambda\left(\Gamma\right)^{\epsilon}]$ 
with 
$\lambda = \nu/ 2^{\epsilon}$,
and the theorem follows.
\end{proof}

%
%
\section{Penalized Lebesgue Constant Estimates and Main Results}
To apply the result of the previous section, we need to estimate the contribution at $x$ from every Lagrange function $\chi_{\xi}$. 
Despite their fast decay, we will never succeed if we do not control the number of centers $\#\Xi$. Normally this would be accomplished with a constraint on the spacing of $\Xi$, a so-called quasi-uniformity condition, which would permit control of $\#\Xi$ by a function of the global density $h$. In our setting this is out of the question. 
Rather, we make the (mostly harmless) assumption that the distance between two nearby centers is 
controlled from below by the local density: $q(\xi) = \min_{\zeta \in \Xi, \zeta\ne \xi} |\xi - \zeta|\ge C \rho(\xi)$.
We call such an assumption {\em weak quasi-uniformity}.
The ratio of weak quasi-uniformity is $c_0 = \max_{\xi\in \Xi}\rho(\xi)/q(\xi)$.
%
%
\begin{theorem}\label{Theorem:Lebesgue}
Assume $f\in C_{c}(\Omega)$ satisfies $\dist\bigl(\supp f,\partial \Omega\bigr) =: r_0$. 
Furthermore, assume $\Xi\subset \Omega$ has local density function 
$\rho$ of precision $\ell \ge m$ and stability $K$, satisfying $1-\epsilon$
slow growth, and weak quasi-uniformity, with ratio $c_0$. 
Then the penalized Lebesgue constant $\L_{\rho,\sigma}$ is bounded, with a constant 
$C(\Omega, \sigma,d,\epsilon,c_0,K,r_0)$ depending only on $\Omega,\sigma,\epsilon, c_0,K$ and $r_0$
(and not  on $\rho$).
\end{theorem}
\begin{proof}
Set $\sigma_1 = \sigma+(1-\epsilon)(m-d/2)$.  
Theorem \ref{Theorem:Lagrange} 
permits us 
to estimate the penalized Lebesgue constant as 
$\L_{\rho,\sigma} 
\le C \sum_{\xi \in\Xi} 
\left(1+\frac{|\xi -x|}{\rho(\xi)}\right)^{\sigma_1}
\exp
  \left[- \lambda
    \left(\frac{\min{|x-\xi|,r_0}}{\rho(\xi)}\right)^\epsilon
  \right]=I+II,$
where 
\begin{eqnarray*}
I &:=& 
\sum_{ \substack{\xi \in\Xi\\ |x-\xi|>r_0}}   
  \left(1+\frac{|\xi -x|}{\rho(\xi)}\right)^{\sigma_1}
  \exp
    \left[- \lambda
      \left(\frac{r_0}{\rho(\xi)}\right)^\epsilon
    \right],\\
II &:=&
 \sum_{ \substack{\xi \in\Xi\\ |x-\xi|\le r_0}}
  \left(1+\frac{|\xi -x|}{\rho(\xi)}\right)^{\sigma_1}
      \exp\left[- \lambda\left(\frac{|x-\xi|}{\rho(\xi)}\right)^\epsilon\right].
\end{eqnarray*}

To treat $I$, we decompose the sum into subsets of $\Xi$ based on the size of $\rho$: 
$$\Xi_k := \{\xi\in \Xi\mid2^{1-k}\le\rho(\xi)\le 2^{-k}\}.$$
Since $\Omega$ is bounded, $\rho$ is bounded above by $r_1$ and, hence,
there is $k_0\in \ints$ so that $\rho(\xi)\le 2^{-k_0}$, for all $\xi$.
By weak quasi-uniformity, there is $\kappa$ so that  for any $k\in \ints$, $2^{-(k +\kappa)}\le q(\xi)$ for all $\xi \in \Xi_k$ 
and $\# \Xi_k \le C_d |\Omega|2^{-d(k +\kappa)}$, which implies
\begin{eqnarray*}
I 
&\le& 
\sum_{k=k_0}^{\infty}
\sum_{\xi\in \Xi_k}
  \left(1+r_1 2^{k+1}\right)^{\sigma_1} 
  \exp\left[- \lambda \left(r_02^{k}\right)^\epsilon\right]\\
&\le&
C
\sum_{k=k_0}^{\infty}
 |\Omega| 2^{(k +\kappa)d}2^{-k\sigma_1} \exp\left[- b 2^{k\epsilon}\right],
\end{eqnarray*}
with $0< b := \lambda r_0^{\epsilon}.$ 
Because $\epsilon>0$, the final sum is bounded, with
a sum depending only on $r_0,\epsilon, r_1,c_0,|\Omega|,d$ and $\sigma$.

To treat $II$, we employ subsets of $\Xi$ defined according to the distance from $x$
$$X_k:=\{\xi\in\Xi\mid2^{k}\rho(x)\le |x-\xi| \le 2^{k+1} \rho(x)\}.$$ 
Slow growth and self-majorization allow us to estimate $\rho(\xi)$ 
from above and below on $\Xi_k$.
Invoking the self-majorization of $\rho$, we see that
$\rho(\xi) \ge \rho(x)2^{-\tau(k+2)-j_0}$ for $\xi \in X_k$,
and for some number $j_0$  depending only on $\epsilon$ for which $2^{-j_0}\le C_{sm}$.
This has two consequences. 
First, 
$\frac{|x-\xi|}{\rho(\xi)}\le 2^{j_0+\tau(k+2)} \frac{|x-\xi|}{\rho(x)}\le 2^{j_1+ k(\tau+1)}$, 
which implies the estimate
$$
\left(1+\frac{|\xi -x|}{\rho(\xi)}\right)^{\sigma_1}
\le C 2^{k(\tau+1)\sigma_1 }.$$
Second, quasi-uniformity ensures that $q(\xi)\ge  c_0\rho(x)2^{-\tau(k+2)-j_0}$, which allows us to estimate $\# X_k$ by 
$(2^{k+1} \rho(x))^d /(c_0\rho(x)2^{-\tau(k+2)-j_0})^d$. Hence
$\# X_k \le C 2^{d k(\tau+1)}$
with a constant depending on $\epsilon,d$ and $c_0$.
The greatest value of $\rho$ in $X_k$  is not larger than
$ \rho(x)  2^{(k+2)(1-\epsilon)}.$ 
It follows that 
$\frac{|x-\xi|}{\rho(\xi)}
\ge b \frac{|x-\xi|}{\rho(x)}2^{k(\epsilon-1)}
\ge b 2^{k\epsilon},$ 
for $b>0$ depending on $\epsilon$.
Thus
\begin{eqnarray*}
II &\le& 
\sum_{k=0}^{\infty}
\sum_{\xi\in \Xi_k}
  (1+\frac{|\xi -x|}{\rho(\xi)})^{\sigma_1}
   \exp\left[- \lambda \left(\frac{|x-\xi|}{\rho(\xi)}\right)^\epsilon\right]\\
&\le&
C
\sum_{k=0}^{\infty}  2^{dk(\tau+1)}
 2^{k(\tau+1)\sigma_1 }  \exp\left[- \lambda' 2^{-k\epsilon^2}\right],
\end{eqnarray*}
with $\lambda' =\lambda b^{\epsilon}>0$. It follows that $II$ is bounded by
a constant depending only on $\epsilon,d,$ and $\sigma$.
\end{proof}
The following corollary applies Lebesgue's lemma to the local approximation results of \cite{DeRo}. Of special interest is the fact that the schemes used by DeVore and Ron to produce  local results are quite abstract and do not lend themselves to direct implementation. For functions of full smoothness $2m$, the 
approximant is determined by modifying an integral representation involving the surface spline. This is worsened when functions of less than full smoothness are treated. In this case a complicated $K$-functional argument is invoked in order to get the result. Interpolation, in contrast, is directly implementable and universal in the sense that it requires only the data to work, and makes no extra requirements of the target function. 
\begin{corollary}
Assume that
$\Xi$ is a set of centers in $\Omega$ having a local density function $\rho$ of precision $\ell$ greater than $2m+1-d$, with $1-\epsilon$ slow growth, weak quasi-uniformity  and, furthermore, that $f\in C(\Omega)$ satisfies $\mathrm{dist}(\supp f,\partial \Omega)>r_0$. 
If $f\in\mathcal{A}^s$, with $\mathcal{A}^s = C^{2m}(\Omega)$  when $s=2m$ or $B_{\infty,\infty}^{s}(\Omega)$ when $0<s<2m$, then 
$$|f(x)-I_{\Xi}f(x)|\le C\; \rho(x)^s \;\|f\|_{\mathcal{A}^{s}(\Omega)}$$
with a constant depending on $\Omega, m,\epsilon,c_0$ and $r_0$.
\end{corollary}
\begin{proof}
We must show that there is a constant $C$ so that, for any target function $f\in \mathcal{A}^s$ and LD $\rho$ satisfying the conditions of the corollary, there exists 
$s_{f,\Xi}\in  S_\Xi$ so that 
\begin{equation}\label{App_prop}
\left\|\frac{f-s_{f,\Xi}}{\rho^s}\right\|_{\infty} \le C .
\end{equation} 
In \cite{DeRo}, an operator 
$T_{\Xi}:C_c^{2m}(\reals^d)\to \spam_{\xi\in\Xi}\phi(\cdot-\xi)$ 
is constructed mapping compactly supported functions to linear combinations of the RBF
$$T_{\Xi} f 
= 
\sum_{\xi\in\Xi} 
\left(
  \int_{\reals^d} 
    \Delta^m f(\alpha) a(\xi, \alpha)
  \dif \alpha
\right) 
\phi(\cdot - \xi),$$ 
with the desired approximation property: 
i.e., $s_{f,\Xi} = T_{\Xi} f$ satisfies (\ref{App_prop}).
For functions of lower smoothness 
(i.e., $f \in B_{\infty,\infty}^{s}$), 
(\ref{App_prop}) holds by first approximating $f$ by a 
smooth function $g$ satisfying the properties 
$|\Delta^mg(x)|\le C \rho(x)^{s-2m} \|f\|_{B_{\infty,\infty}^s}$ 
and 
$|f(x)-g(x)|\le C \rho(x)^s \|f\|_{B_{\infty,\infty}^s}$ 
and then selecting $s_{f,\Xi} := T_{\Xi} g$. 
In either case, the approximant 
$s_{f,\Xi} $ has coefficients determined by an expression of the form
$A_{\xi} =\int_{\reals^d} \Delta^m F(\alpha) a(\xi, \alpha)\dif \alpha$, 
where $F$ has compact support. 
Utilizing the polynomial reproduction of $a$, we see that, for 
$p \in \Pi_{\min(\ell,2m-1)}$,
$$\sum_{\xi\in\Xi}A_{\xi} p(\xi) 
= 
\int_{\reals^d}\Delta^mF(\alpha) a(\xi, \alpha) p(\xi)\dif \alpha 
= 
\int_{\reals^d} \Delta^m F(\alpha)  p(\alpha)\dif \alpha 
= 0,$$
by applying Green's identity
\begin{eqnarray*} 
\int_{B(0,R)}\Delta^m F(\alpha) \, p(\alpha)\dif \alpha 
& = & 
\int_{B(0,R)} F(\alpha)  \Delta^mp(\alpha)\dif \alpha \\
&\mbox{}-&
\sum_{j=0}^{2m-1}(-1)^j
 \int_{\partial B(0,R)} 
    \lambda_{j+1} F (\alpha) \lambda_{2m-j} p(\alpha)
 \dif \sigma(\alpha)
\end{eqnarray*}
for a ball $B(0,R)$ of appropriate radius
(the boundary operators $\lambda_j$ are 
$\lambda_j  := \Delta^{(j-1)/2} $ or   
$\lambda_j  := D_n \Delta^{(j-2)/2} $ 
for odd and even $j$, respectively). 
Thus, $s_{f,\Xi}\in S_{\Xi}$ and the corollary follows from Lebesgue's lemma.
\end{proof}
\bibliographystyle{amsplain}
\bibliography{LL_test}
\end{document}